\documentclass[12pt]{amsart}
\usepackage{graphicx}
\usepackage{amssymb}
\usepackage{cite}
\usepackage{amsmath}
\usepackage{latexsym}
\usepackage{amscd}
\usepackage{amsthm}
\usepackage{mathrsfs}
\usepackage{url}
\usepackage{hyperref}
\usepackage{verbatim}

\newtheorem{thm}{Theorem}[section]

\theoremstyle{definition}

\theoremstyle{remark}

\numberwithin{equation}{section}
\begin{document}

\title[Compactness of minimal surfaces]{Compactness for minimal surfaces
  with injectivity radius bounded from below}

\author{Luan Figueiredo}
\address{Universidade Federal de Minas Gerais (UFMG),
  Departamento de Matem\'{a}tica, Caixa Postal 702,
  30123-970, Belo Horizonte, MG, Brazil}
\email{luanmat@ufmg.br}

\author{Rosivaldo Gon\c{c}alves}
\address{Universidade Estadual de Montes Claros (Unimontes),
  Departamento de Matem\'{a}tica, 39401-089, Montes Claros, MG, Brazil}
\email{rosivaldo.goncalvees@unimontes.br}

\thanks{The authors were partially supported by CNPq and CAPES/Brazil grants.}
\date{\today}
\maketitle

\begin{abstract}
We prove a compactness theorem for the space of closed embedded minimal
surfaces with area bounded from above and injectivity radius bounded from
below in a closed Riemannian $3$-manifold. This result is a variant of the
Choi--Schoen compactness theorem in which the genus bound is replaced by a
lower bound on the injectivity radius of the surface.
\end{abstract}

\section{Introduction}

The study of sequences of minimal surfaces in $3$-manifolds is a central
topic in differential geometry. The foundational work of Allard~\cite{Al}
establishes weak convergence results for sequences of minimal surfaces via
the framework of varifolds. Among the most celebrated compactness results
for closed minimal surfaces is the Choi--Schoen Compactness
Theorem~\cite{Choi-1}, which asserts that any sequence of closed embedded
minimal surfaces in a closed Riemannian $3$-manifold with positive Ricci
curvature and uniformly bounded genus has a convergent subsequence in the
$C^\infty$ topology. White subsequently generalized this to stationary points
of arbitrary elliptic functionals on the space of embeddings of a compact
surface, under an additional area bound.

Both the genus bound and the area bound heuristically control the injectivity
radius of the surfaces. The following classical examples illustrate how
compactness can fail when the injectivity radius degenerates.

\subsection{Example: Catenoid}
Let $\Sigma_c$ denote the standard catenoid in $\mathbb{R}^3$ and set
$\Sigma_i = \frac{1}{i}\Sigma_c$. For every point $x \in \Sigma_i$ one has
\[
  |x^{\perp}|^2\,|A_{\Sigma_i}(x)|^2 \leq 2.
\]
In particular, $\Sigma_i$ pinches off at the origin and
$|A_{\Sigma_i}(x)| \to \infty$ for every $x \in \Sigma_i$ as $i\to\infty$.
The sequence $\{\Sigma_i\}$ converges weakly, as varifolds, to the plane
with multiplicity two. Here compactness fails because the neck of the catenoid
pinches off: the injectivity radius of $\Sigma_i$ tends to zero.

\subsection{Example: Helicoid}
The helicoid is the complete, simply connected minimal surface
$\Sigma_h \subset \mathbb{R}^3$ parametrized by
\[
  \varphi(u,v)= (u\cos v,\; u\sin v,\; v), \qquad (u,v)\in \mathbb{R}^2.
\]
Set $\Sigma_i = \frac{1}{i}\Sigma_h$. While $\Sigma_h$ is invariant under
vertical translations by $2\pi m$, the surface $\Sigma_i$ is invariant under
vertical translations by $\frac{2\pi m}{i}$. One can verify that $\{\Sigma_i\}$
converges smoothly away from the vertical axis to a foliation by horizontal
planes, while the curvature blows up along the vertical axis. Compactness
fails here due to the absence of local area bounds, again a consequence of
the injectivity radius tending to zero.

\medskip

These examples suggest that a lower bound on the injectivity radius of the
surfaces is the natural condition needed to recover compactness. Our main
result confirms this.

\begin{thm}\label{main}
Let $M^3$ be a closed Riemannian $3$-manifold and let $S$ denote the set of
closed embedded minimal surfaces in $M^3$. For constants $A_0, i_0 > 0$,
define
\[
  \mathcal{C}(A_0,i_0) := \bigl\{\Sigma \in S \;:\;
    \operatorname{Area}(\Sigma) \leq A_0 \text{ and }
    \operatorname{inj}(\Sigma) \geq i_0 \bigr\}.
\]
Then $\mathcal{C}(A_0,i_0)$ is compact in the $C^{\infty}$ topology.
More precisely, any sequence in $\mathcal{C}(A_0,i_0)$ has a subsequence
converging in $C^k$ on compact subsets to a surface in
$\mathcal{C}(A_0,i_0)$, for every $k \geq 2$.
\end{thm}

The proof uses crucially that the surfaces are two-dimensional, as the
argument relies on the Gauss--Bonnet theorem and on the classification result
of L\'opez--Ros~\cite{LR}.

In higher dimensions, Sharp~\cite{b-sharp} proved a smooth compactness result
for closed embedded orientable minimal hypersurfaces with bounded index and
volume in closed manifolds $M^{n+1}$ with positive Ricci curvature and
$2 \leq n \leq 6$, generalizing the Choi--Schoen theorem. It is natural to
ask whether Theorem~\ref{main} holds in higher dimensions as well. Another
natural generalization, to be addressed in forthcoming work, is to obtain
analogous compactness results for stationary points of general elliptic
functionals defined on embeddings of compact surfaces, with the area
functional replaced by an arbitrary such functional.

\section{Proof of Theorem~\ref{main}}

\begin{proof}
Let $\{\Sigma_n\}$ be a sequence of closed embedded minimal surfaces
satisfying
\[
  \operatorname{Area}(\Sigma_n) \leq A_0
  \quad \text{and} \quad
  \operatorname{inj}(\Sigma_n) \geq i_0.
\]
We claim there exists a constant $C > 0$ such that
$\sup_{\Sigma_n} |A_n| \leq C$ for all $n$,
where $A_n$ denotes the second fundamental form of $\Sigma_n$.

Suppose for contradiction that
$\lambda_n := \sup_{\Sigma_n} |A_n| \to \infty$.
Choose base points $p_n \in \Sigma_n$ such that $|A_n|(p_n) = \lambda_n$.
Consider the rescaled surfaces $\Sigma_n' := \lambda_n \Sigma_n$ inside the
geodesic ball $\bigl(B_{\lambda_n i_0/2}(p_n),\, \lambda_n^2 g\bigr)$.
By construction,
\[
  \sup_{\Sigma_n'} |A_{\Sigma_n'}| \leq 1
  \quad \text{and} \quad
  |A_{\Sigma_n'}|(p_n) = 1.
\]
As $n \to \infty$, the rescaled metrics $\lambda_n^2 g$ converge to the flat
metric on $\mathbb{R}^3$.

\smallskip\noindent\textbf{Convergence.}
If $\{\Sigma_n'\}$ satisfies local area bounds, i.e.,
$|\Sigma_n' \cap B_R(p_n)| \leq C_R$ for some constant $C_R$ depending only
on $R$, then standard compactness arguments yield a subsequence converging to
a properly embedded minimal surface $\Sigma_0 \subset \mathbb{R}^3$.
Otherwise, $\{\Sigma_n'\}$ converges to a minimal lamination of $\mathbb{R}^3$
containing a leaf $\Sigma_0$.
In either case, $\Sigma_0$ satisfies
\begin{equation}\label{non flat}
  \sup_{\Sigma_0} |A_{\Sigma_0}| = |A_{\Sigma_0}|(0) = 1,
\end{equation}
so $\Sigma_0$ is not totally geodesic.

Since $\operatorname{inj}(\Sigma_n) \geq i_0$ and $\lambda_n \to \infty$,
the injectivity radius of $\Sigma_0$ at the origin satisfies
\[
  \operatorname{inj}(\Sigma_0) = \lim_{n \to \infty} \lambda_n\,
  \operatorname{inj}(\Sigma_n) = \infty.
\]
In particular $\Sigma_0$ is topologically a disk.
Moreover, $\Sigma_0$ is properly embedded by a result of
Rosenberg~\cite{Rosenberg}.

\smallskip\noindent\textbf{Bounded total curvature.}
We now show that $\Sigma_0$ has finite total curvature.
Since $\operatorname{inj}(\Sigma_n) \geq i_0$, the intrinsic geodesic ball
$B_{i_0/2}^{\Sigma_n}(p_n)$ is topologically a disk.
Applying the Gauss--Bonnet theorem to this disk (see, e.g.,
\cite[Chapter~2]{CM2011}), for each $t \in (0, i_0)$ we have
\[
  \int_{\partial B_t^{\Sigma_n}} k_g \;=\; 2\pi - \int_{B_t^{\Sigma_n}} K_{\Sigma_n},
\]
where $k_g$ denotes the geodesic curvature of $\partial B_t^{\Sigma_n}$ in
$\Sigma_n$. Integrating over $t \in [0, \rho]$ gives
\begin{equation}\label{equation1}
  |\partial B_\rho^{\Sigma_n}| - 2\pi\rho
  = -\int_0^\rho \int_{B_t^{\Sigma_n}} K_{\Sigma_n}\, dt.
\end{equation}
Integrating~\eqref{equation1} over $\rho \in [0, i_0]$ and applying the
co-area formula yields
\begin{equation}\label{equation2}
  |B_{i_0}^{\Sigma_n}(p_n)| - \pi i_0^2
  = -\int_0^{i_0}\int_0^\rho \int_{B_t^{\Sigma_n}} K_{\Sigma_n}\, dt\, d\rho.
\end{equation}
Since $\Sigma_n$ is minimal, the Gauss equation reads
\[
  \overline{K}_M(T\Sigma_n) = K_{\Sigma_n} + \tfrac{1}{2}|A_n|^2,
\]
where $\overline{K}_M(T\Sigma_n)$ denotes the sectional curvature of $M$
along $\Sigma_n$. Substituting into~\eqref{equation2},
\begin{align*}
  |B_{i_0}^{\Sigma_n}(p_n)| - \pi i_0^2
  &= -\int_0^{i_0}\int_0^\rho \int_{B_t^{\Sigma_n}}
    \overline{K}_M(T\Sigma_n)\,dt\,d\rho \\
  &\quad + \frac{1}{2}\int_0^{i_0}\int_0^\rho
    \int_{B_t^{\Sigma_n}} |A_n|^2\,dt\,d\rho \\
  &\geq -\int_0^{i_0}\int_0^\rho \int_{B_t^{\Sigma_n}}
    \overline{K}_M(T\Sigma_n)\,dt\,d\rho
    + \frac{i_0^2}{8}\int_{B_{i_0/2}^{\Sigma_n}(p_n)} |A_n|^2,
\end{align*}
where the last inequality uses the co-area formula together with the fact
that $B_{i_0/2}^{\Sigma_n}(p_n)\subset B_t^{\Sigma_n}$ for all
$t \geq i_0/2$.
Since $M^3$ is compact, there exists $K_0 > 0$ with
$|\overline{K}_M| \leq K_0$, and since $\operatorname{Area}(\Sigma_n) \leq A_0$,
we obtain
\[
  \frac{i_0^2}{8}\int_{B_{i_0/2}^{\Sigma_n}(p_n)} |A_n|^2
  \;\leq\; A_0 - \pi i_0^2 + K_0\,\frac{i_0^2}{2}\,A_0.
\]
Hence there exists a constant $C_1 = C_1(i_0, A_0, K_0) > 0$ such that
\begin{equation}\label{bounded total curvature}
  \int_{B_{i_0/2}^{\Sigma_n}(p_n)} |A_n|^2 \;\leq\; C_1.
\end{equation}
Since the left-hand side of~\eqref{bounded total curvature} is scale
invariant, passing to the limit gives
\[
  \int_{\Sigma_0} |A_{\Sigma_0}|^2 < \infty.
\]

\smallskip\noindent\textbf{Conclusion.}
By the theorem of L\'opez--Ros~\cite{LR}, the only properly embedded minimal
surface in $\mathbb{R}^3$ of genus zero with finite total curvature is the
totally geodesic plane. Since $\Sigma_0$ is a disk (hence genus zero) and
has finite total curvature, it must be a plane---contradicting~\eqref{non flat}.

Therefore $\sup_{\Sigma_n} |A_n| \leq C$ for some uniform constant $C > 0$,
and standard elliptic theory implies that every sequence in
$\mathcal{C}(A_0, i_0)$ has a subsequence converging smoothly to a surface
in $\mathcal{C}(A_0, i_0)$.
\end{proof}


\bibliographystyle{plain}

\end{document}